\documentclass{amsart}
\usepackage[all]{xy}
\usepackage{yfonts}
\usepackage{color} 
\usepackage{amsthm}
\usepackage{amssymb}
\usepackage[colorlinks=true]{hyperref}
\usepackage{tikz}
\usetikzlibrary{matrix,arrows}

\numberwithin{equation}{section}

\newtheorem{theorem}[equation]{Theorem} 
\newtheorem*{theorem*}{Theorem}
\newtheorem{lemma}[equation]{Lemma}
\newtheorem{proposition}[equation]{Proposition}

\newtheorem*{corollary*}{Corollary}

\theoremstyle{remark}
\newtheorem{definition}[equation]{Definition}

\newtheorem{example}[equation]{Example}

\newtheorem{notation}[equation]{Notation}

\theoremstyle{remark}
\newtheorem{remark}[equation]{Remark}

\newcommand{\cA}{{\mathcal A}}
\newcommand{\cB}{{\mathcal B}}
\newcommand{\cC}{{\mathcal C}}
\newcommand{\cD}{{\mathcal D}}
\newcommand{\cE}{{\mathcal E}}
\newcommand{\cF}{{\mathcal F}}
\newcommand{\cG}{{\mathcal G}}
\newcommand{\cH}{{\mathcal H}}

\newcommand{\cO}{{\mathcal O}}

\newcommand{\Spt}{\mathrm{Spt}}

\newcommand{\add}{\mathsf{add}}

\newcommand{\bbA}{\mathbb{A}}

\newcommand{\bbG}{\mathbb{G}}

\newcommand{\bbP}{\mathbb{P}}

\newcommand{\bbZ}{\mathbb{Z}}

\newcommand{\bfR}{\mathbf{R}}
\newcommand{\bfL}{\mathbf{L}}

\DeclareMathOperator{\id}{id}

\DeclareMathOperator{\Mod}{Mod}

\DeclareMathOperator{\Fun}{Fun} 

\newcommand{\SmProj}{\mathsf{SmProj}} 
\newcommand{\Sep}{\mathsf{Sep}} 
\newcommand{\NChow}{\mathsf{NChow}} 

\newcommand{\bbK}{I\mspace{-6.mu}K}

\newcommand{\dgcat}{\mathsf{dgcat}}

\newcommand{\perf}{\mathrm{perf}}

\newcommand{\dg}{\mathsf{dg}}

\newcommand{\Hom}{\mathrm{Hom}}
\newcommand{\End}{\mathrm{End}}

\newcommand{\rep}{\mathrm{rep}}

\newcommand{\dgHo}{\mathsf{H}^0}

\newcommand{\Ho}{\mathrm{Ho}}

\newcommand{\Hmo}{\mathsf{Hmo}}
\newcommand{\op}{\mathrm{op}}

\newcommand{\too}{\longrightarrow}

\newcommand{\ie}{\textsl{i.e.}\ }
\newcommand{\eg}{\textsl{e.g.}}

\begin{document}

\title[Invariants of toric and homogeneous varieties via NC motives]{Additive invariants of toric and \\twisted projective homogeneous varieties\\ via noncommutative motives}

\author{Gon{\c c}alo~Tabuada}

\address{Gon{\c c}alo Tabuada, Department of Mathematics, MIT, Cambridge, MA 02139, USA}
\email{tabuada@math.mit.edu}
\urladdr{http://math.mit.edu/~tabuada/}
\thanks{The author was partially supported by the NEC Award-2742738}
\date{\today}


\abstract{I. Panin proved in the nineties that the algebraic $K$-theory of twisted projective homogeneous varieties can be expressed in terms of central simple algebras. Later, Merkurjev and Panin described the algebraic $K$-theory of toric varieties as a direct summand of the algebraic $K$-theory of separable algebras. In this article, making use the recent theory of noncommutative motives, we extend Panin and Merkurjev-Panin computations from algebraic $K$-theory to every additive invariant. As a first application, we fully compute the cyclic homology (and all its variants) of twisted projective homogeneous varieties. As a second application, we show that the noncommutative motive of a twisted projective homogeneous variety is trivial if and only if the Brauer classes of the associated central simple algebras are trivial. Along the way we construct a fully-faithful $\otimes$-functor from Merkurjev-Panin's motivic category to Kontsevich's category of noncommutative Chow motives, which is of independent interest.}}

\maketitle 
\vskip-\baselineskip
\vskip-\baselineskip
\vskip \baselineskip

\dedicatory{{\it Dedicated to the memory of Daniel Kan.}}
\section{Introduction}\label{sec:introduction}
\subsection*{Algebraic $K$-theory of twisted projective homogeneous varieties}
Let $G$ be a split semisimple algebraic group over a field $k$, $P \subset G$ a parabolic subgroup, and $\gamma: \mathfrak{g}:=\mathrm{Gal}(k_{\mathrm{sep}}/k) \to G(k_{\mathrm{sep}})$ a $1$-cocycle. Out of this data one can construct the projective homogeneous variety $\cF:=G/P$ as well as its twisted form ${}_\gamma \cF$. Let $\widetilde{G}$ and $\widetilde{P}$ be the universal covers of $G$ and $P$, $R(\widetilde{G})$ and $R(\widetilde{P})$ the associated representation rings, $n$ the index $[W(\widetilde{G}):W(\widetilde{P})]$ of the Weyl groups, $\widetilde{Z}$ the center of $\widetilde{G}$, and finally $\mathsf{Ch}:=\Hom(\widetilde{Z},\bbG_m)$ the character group. Under these notations, Panin proved in \cite[Thm.~4.2]{Panin} that every $\mathsf{Ch}$-homogeneous basis $\rho_1, \ldots, \rho_n$ of $R(\widetilde{P})$ over $R(\widetilde{G})$ gives rise to an isomorphism
\begin{eqnarray}\label{eq:iso-main}
\bigoplus_{i=1}^n K_\ast(A_{\chi(i),\gamma}) \stackrel{\sim}{\too} K_\ast({}_\gamma \cF)\,,
\end{eqnarray}
where $A_{\chi(i),\gamma}$ stands for the Tits' central simple algebra associated to $\rho_i$. Panin's computation \eqref{eq:iso-main} is a landmark in algebraic $K$-theory. It generalizes previous results of Grothendieck \cite{Grothendieck} on flag varieties, of Quillen \cite{Quillen} on Severi-Brauer varieties, of Swan \cite{Swan} and Kapranov \cite{Kapranov} on quadrics hypersurfaces, and of Levine-Srinivas-Weyman \cite{Levine} on twisted Grassmann varieties.
\subsection*{Algebraic $K$-theory of toric varieties}
Let $S$ be a reductive algebraic group over $k$ and $A$ the associated division separable algebra $\prod_\rho A_\rho := \prod_\rho \End_S(W_\rho)$, where the product is taken over all irreducible representation $\rho:S \to \mathrm{GL}(W_\rho)$. Given an $S$-torsor $\pi:U \to X$ over a smooth projective variety $X$, assume that there exists an $S$-equivariant open imbedding of $U$ into an affine space on which $S$ acts linearly. Under these assumptions, Merkurjev and Panin proved in \cite[Thm.~4.2]{MP} that $K_\ast(X)$ is a direct summand of $K_\ast(A)$. Examples include toric models (see Example \ref{example:toric}) and more generally toric varieties (see Remark~\ref{rk:toric}). 
\subsection*{Additive invariants}
A {\em dg category} $\cA$, over a field $k$, is a category enriched over complexes of $k$-vector spaces; see \S\ref{sec:dg}. Let $\dgcat$ be the category of (small) dg categories. Every (dg) algebra $A$ gives naturally rise to a dg category with a single object. Another source of examples is provided by schemes since the derived category of perfect complexes $\perf(X)$ of every quasi-compact quasi-separated scheme $X$ admits a canonical dg enhancement\footnote{When $X$ is quasi-projective this dg enhancement is unique; see Lunts-Orlov \cite[Thm.~2.12]{LO}.} $\perf_\dg(X)$; see Keller \cite[\S4.6]{ICM-Keller}.

Given a dg category $\cA$, let $T(\cA)$ be the dg category of pairs $(i,x)$, where $i\in \{1,2\}$ and $x \in \cA$. The complex of morphisms in $T(\cA)$ from $(i,x)$ to $(i',x')$ is given by $\cA(x,x')$ if $i \leq i'$ and is zero otherwise. Composition is induced from $\cA$. Intuitively speaking, $T(\cA)$ ``dg categorifies'' the notion of upper triangular matrix.  Note that we have two inclusion dg functors $i_1: \cA \hookrightarrow T(\cA)$ and $i_2: \cA \hookrightarrow T(\cA)$.
\begin{definition}\label{def:additive}
Let $E:\dgcat\to \mathsf{D}$ be a functor with values in an additive category. We say that $E$ is an {\em additive invariant} if it satisfies the following two conditions:
\begin{itemize}
\item[(i)] it sends {\em derived Morita equivalences} (see \S\ref{sec:dg}) to isomorphisms;
\item[(ii)] given any dg category $\cA$, the inclusion dg functors induce an isomorphism\footnote{Condition (ii) can be equivalently formulated in terms of semi-orthogonal decompositions in the sense of Bondal-Orlov; see \cite[Thm.~6.3(4)]{IMRN}.}
$$ [E(i_1)\,\,E(i_2)]:E(\cA) \oplus E(\cA) \stackrel{\sim}{\too} E(T(\cA))\,.$$
\end{itemize}
\end{definition}
Examples of additive invariants include:
\begin{itemize}
\item[(i)] The mixed complex functor $C:\dgcat \to \cD(\Lambda)$ with values in the derived category of mixed complexes; see Keller \cite[\S1.5 Thm. b) and \S1.12]{Exact}.
\item[(ii)] The Hochschild homology functor $HH: \dgcat \to \cD(k)$ (with values in the derived category of $k$), the cyclic homology functor $HC:\dgcat \to \cD(k)$, the periodic cyclic homology functor $HP:\dgcat \to \cD_{\bbZ/2}(k)$ (with values in the derived category of $\bbZ/2$-graded complexes of $k$-vector spaces), and the negative cyclic homology functor $HN:\dgcat \to \cD(k)$; see Keller \cite[\S2.2]{Exact2}.
\item[(iii)] The connective algebraic $K$-theory functor $K:\dgcat \to \Ho(\Spt)$ with values in the homotopy category of spectra; see Thomason-Trobaugh \cite[Thm.~1.9.8]{TT} and Waldhausen \cite[Thm.~1.4.2]{Wald}.
\item[(iv)] The mod-$l$ algebraic $K$-theory functor $K(-;\bbZ/l): \dgcat \to \Ho(\Spt)$, with $l$ an integer $\geq 2$; see \cite[\S3]{Galois}. 
\item[(v)] The nonconnective algebraic $K$-theory functor $\bbK:\dgcat \to \Ho(\Spt)$; see Schlichting \cite[\S7 Thm. 4 and \S12.3 Prop. 3]{Negative}.
\item[(vi)] The homotopy algebraic $K$-theory functor $KH:\dgcat \to \Ho(\Spt)$; see \cite[Prop.~3.3]{Galois}.
\item[(vii)] The topological Hochschild homology functor $THH:\dgcat \to \Ho(\Spt)$ and the topological cyclic homology functor $TC:\dgcat \to \Ho(\Spt)$; see Blumberg-Mandell \cite[Prop. 3.10 and Thm.~10.8]{BM} and \cite[Prop.~8.9]{MacLane}.
\item[(viii)] The universal additive invariant $U :\dgcat \to \Hmo_0$ with values in the category of noncommutative motives; see \S\ref{sec:NCmotives}. 
\end{itemize}
When applied to $A$, respectively to $\perf_\dg(X)$, the above additive invariants (i)-(vii) reduce to the classical invariants of (dg) algebras, respectively of schemes: consult \cite[Thm.~5.2]{ICM-Keller} for (i)-(ii), \cite[Thm.~5.1]{ICM-Keller} for (iii) and (v), \cite[Example~2.13]{Products} for (iv), \cite[Prop.~2.3]{Fundamental} for (vi), and \cite[Thm.~1.3]{BM} for (vii). For this reason, we will write $E(X)$ instead of $E(\perf_\dg(X))$.
\section{Statement of results}\label{sec:results}
\begin{theorem}\label{thm:main}
\begin{itemize}
\item[(i)] Let $G,P, \gamma$ be as above and $E:\dgcat \to \mathsf{D}$ an additive invariant. Under these assumptions, every $\mathsf{Ch}$-homogeneous basis $\rho_1, \ldots, \rho_n$ of $R(\widetilde{P})$ over $R(\widetilde{G})$ gives rise to an isomorphism
\begin{equation}\label{eq:iso-main-2}
\bigoplus_{i=1}^n E(A_{\chi(i),\gamma}) \stackrel{\sim}{\too} E({}_\gamma \cF)\,.
\end{equation}
\item[(ii)] Let $S, \pi, X$ be as above and $E:\dgcat \to \mathsf{D}$ an additive invariant. Under these assumptions, $E(X)$ is a direct summand of $E(A)$. 
\end{itemize}
\end{theorem}
\begin{remark}[Quasi-split case]\label{rk:quasi-split}
When $G$ is a {\em quasi}-split algebraic group, Panin proved in \cite[Thm.~12.4]{Panin} that a computation similar to \eqref{eq:iso-main} also holds. In this generality, the algebras $A_{\chi(i),\gamma}$ are no longer central simple but only separable. The analogue of isomorphism \eqref{eq:iso-main-2} (with exactly the same proof) also holds.
\end{remark}
The proof of Theorem \ref{thm:main} is based on the recent theory of noncommutative motives (see \S\ref{sec:NCmotives}) and on a fully-faithful $\otimes$-functor from Merkurjev-Panin's motivic category to Kontsevich's category of noncommutative Chow motives (see \S\ref{sec:MP-K}). 
\section{Applications}\label{sec:applications}
\subsection*{Twisted projective homogeneous varieties}
By applying \eqref{eq:iso-main-2} to the above examples (i)-(viii) of additive invariants we obtain the (concrete) isomorphisms:
$$
\begin{array}{lcl}
\bigoplus_{i=1}^n C(A_{\chi(i),\gamma})\simeq C({}_\gamma \cF) && \\
\bigoplus_{i=1}^n HH(A_{\chi(i),\gamma})\simeq HH({}_\gamma \cF)  && \bigoplus_{i=1}^n HH_\ast(A_{\chi(i),\gamma})\simeq HH_\ast({}_\gamma \cF) \\
\bigoplus_{i=1}^n HC(A_{\chi(i),\gamma})\simeq HC({}_\gamma \cF)  && \bigoplus_{i=1}^n HC_\ast(A_{\chi(i),\gamma})\simeq HC_\ast({}_\gamma \cF) \\
\bigoplus_{i=1}^n HP(A_{\chi(i),\gamma})\simeq HP({}_\gamma \cF)  && \bigoplus_{i=1}^n HP_\ast(A_{\chi(i),\gamma})\simeq HP_\ast({}_\gamma \cF) \\
\bigoplus_{i=1}^n HN(A_{\chi(i),\gamma})\simeq HN({}_\gamma \cF)  && \bigoplus_{i=1}^n HN_\ast(A_{\chi(i),\gamma})\simeq HN_\ast({}_\gamma \cF) \\
\bigoplus_{i=1}^n K(A_{\chi(i),\gamma})\simeq K({}_\gamma \cF)  && \bigoplus_{i=1}^n K_\ast(A_{\chi(i),\gamma})\simeq K_\ast({}_\gamma \cF) \\
\bigoplus_{i=1}^n K(A_{\chi(i),\gamma}; \bbZ/l)\simeq K({}_\gamma \cF; \bbZ/l)  && \bigoplus_{i=1}^n K_\ast(A_{\chi(i),\gamma};\bbZ/l)\simeq K_\ast({}_\gamma \cF;\bbZ/l) \\
\bigoplus_{i=1}^n \bbK(A_{\chi(i),\gamma})\simeq \bbK({}_\gamma \cF)  && \bigoplus_{i=1}^n \bbK_\ast(A_{\chi(i),\gamma})\simeq \bbK_\ast({}_\gamma \cF) \\
\bigoplus_{i=1}^n KH(A_{\chi(i),\gamma})\simeq KH({}_\gamma \cF)  && \bigoplus_{i=1}^n KH_\ast(A_{\chi(i),\gamma})\simeq KH_\ast({}_\gamma \cF) \\
\bigoplus_{i=1}^n THH(A_{\chi(i),\gamma})\simeq THH({}_\gamma \cF)  && \bigoplus_{i=1}^n THH_\ast(A_{\chi(i),\gamma})\simeq THH_\ast({}_\gamma \cF) \\
\bigoplus_{i=1}^n TC(A_{\chi(i),\gamma})\simeq TC({}_\gamma \cF)  && \bigoplus_{i=1}^n TC_\ast(A_{\chi(i),\gamma})\simeq TC_\ast({}_\gamma \cF) \\
\bigoplus_{i=1}^n U(A_{\chi(i),\gamma}) \simeq U({}_\gamma \cF) &&
\end{array}
$$
Note that the left-hand-side isomorphisms enhance the right-hand-side ones. In particular, we obtain a spectral enhancement $\bigoplus_{i=1}^n K(A_{\chi(i),\gamma})\simeq K({}_\gamma \cF)$ of Panin's original computation \eqref{eq:iso-main}. In what concerns cyclic homology (and all its variants), we have the following complete computation:
\begin{theorem}\label{thm:computation}
Under the assumptions of Theorem~\ref{thm:main}(i), we have:
$$
\begin{array}{lcl}
\bigoplus_{i=1}^n C(k) \simeq C({}_\gamma \cF) && \\
\bigoplus _{i=1}^n HH(k) \simeq HH({}_\gamma \cF) && \bigoplus _{i=1}^n HC(k) \simeq HC({}_\gamma \cF) \\
\bigoplus _{i=1}^n HP(k) \simeq HP({}_\gamma \cF) && \bigoplus _{i=1}^n HN(k) \simeq HN({}_\gamma \cF)\,. 
\end{array}
$$
\end{theorem}
As a consequence of Theorem~\ref{thm:computation} one obtains the isomorphisms:
\begin{eqnarray*}
HH_j({}_\gamma\cF) \simeq \left\{ \begin{array}{ll}
\oplus_{i=1}^n k &   j=0 \\
0 &  \text{otherwise} \\
\end{array} \right. && 
HC_j({}_\gamma \cF) \simeq \left\{ \begin{array}{ll}
\oplus_{i=1}^n k & j\geq 0 \,\, \text{even}  \\
0 &  \text{otherwise} \\
\end{array} \right.
\end{eqnarray*}
\begin{eqnarray*}
HP_j({}_\gamma \cF) \simeq \left\{ \begin{array}{ll}
\oplus_{i=1}^n k &  j \,\,\text{even} \\
0 & n \,\,\text{odd} \\
\end{array} \right. && 
HN_j({}_\gamma \cF) \simeq \left\{ \begin{array}{ll}
\oplus_{i=1}^n  k &  j\leq 0 \,\, \text{even}  \\
0 &  \text{otherwise} \\
\end{array} \right.\,.
\end{eqnarray*}

Intuitively speaking, Theorem \ref{thm:computation} shows that cyclic homology (and all its variants) only measures the index $n:=[W(\widetilde{G}):W(\widetilde{P})]$ of the Weyl groups. Under the following restrictions, the same holds for every additive invariant. 
\begin{theorem}\label{thm:coefficients}
The above isomorphism \eqref{eq:iso-main-2} reduces to $\bigoplus_{i=1}^n E(k) \simeq E({}_\gamma \cF)$ whenever $\mathsf{D}$ is $\bbZ[\frac{1}{r}]$-linear, with $r:=\prod_{i=1}^n \mathrm{degree}(A_{\chi(i),\gamma})$.
\end{theorem}
The richest additive invariant is the universal one. In this case we have the following optimal result:
\begin{theorem}\label{thm:Brauer}
Under the assumptions of Theorem~\ref{thm:main}(i), we have an isomorphism $\bigoplus_{i=1}^n U(k) \simeq U({}_\gamma \cF)$ if and only if the Brauer classes $[A_{\chi(i),\gamma}]$ are trivial. 
\end{theorem}
\begin{example}[Severi-Brauer varieties]
Let $G=PGL_n$. In this case, $\widetilde{G}=SL_n$, $\widetilde{Z}\simeq\mu_n$ and $\mathsf{Ch}\simeq\bbZ/n\bbZ$. Consider the following parabolic subgroup
$$ \widetilde{P} := \{\left( \begin{array}{cc}
a & b  \\
0 & c 
\end{array} \right) | \,a \cdot \mathrm{det}(c) =1 \}\subset SL_n\,,$$
where $a \in k^\times$ and $c \in GL_{n-1}$. The associated projective homogeneous variety is $\cF:=G/P \simeq \widetilde{G}/\widetilde{P} = \bbP^{n-1}$. Given a $1$-cocycle $\mathfrak{g} \to PGL_n(k_{\mathrm{sep}})$ and an additive invariant $E:\dgcat \to \mathsf{D}$, we conclude from Panin \cite[\S10.2]{Panin} and \eqref{eq:iso-main-2} that 
\begin{equation}\label{eq:iso-SV}
E(k) \oplus E(A_\gamma) \oplus \cdots \oplus E(A_\gamma^{\otimes n-1}) \stackrel{\sim}{\too} E({}_\gamma \bbP^{n-1})\,,
\end{equation}
where $A_\gamma$ is the central simple algebra of degree $n$ associated to $\gamma$. Thanks to Theorem~\ref{thm:coefficients}, \eqref{eq:iso-SV} reduces to $\bigoplus_{i=1}^n E(k)\simeq E({}_\gamma \bbP^{n-1})$ whenever $\mathsf{D}$ is $\bbZ[\frac{1}{n}]$-linear. 
\end{example}
\begin{example}[Twisted Grassmann varieties]
Let $G:=PGL_n$ as above. Fix a number $1 \leq m \leq n-1$ and consider the parabolic subgroup
$$ \widetilde{P} := \{\left( \begin{array}{cc}
a & b  \\
0 & c 
\end{array} \right) | \,\mathrm{det}(a) \cdot \mathrm{det}(c) =1 \} \subset SL_n\,,$$
where $a \in GL_m$ and $c\in GL_{n-m}$. The associated projective homogeneous variety is $\cF:=G/P \simeq \widetilde{G}/\widetilde{P}= \mathrm{Gr}(m,n)$. Given a $1$-cocycle $\gamma:\mathfrak{g} \to PGL_n(k_{\mathrm{sep}})$ and an additive invariant $E:\dgcat \to \mathsf{D}$, one concludes from \cite[\S10.2]{Panin} and \eqref{eq:iso-main-2} that 
\begin{equation}\label{eq:iso-Grass}
\bigoplus_\alpha E(A_\gamma^{\otimes d(\alpha)}) \stackrel{\sim}{\too} E({}_\gamma \mathrm{Gr}(m,n))\,,
\end{equation}
where $\alpha$ runs over all sequences $\alpha_1, \ldots, \alpha_m$ such that $n-m > \alpha_1 -1 > \cdots > \alpha_m - m > -m$, and $d(\alpha):= \alpha_1 + \cdots + \alpha_m$. Thanks to Theorem~\ref{thm:coefficients},  \eqref{eq:iso-Grass} reduces to $\bigoplus_{\alpha} E(k)\simeq E({}_\gamma \mathrm{Gr}(m,n))$ whenever $\mathsf{D}$ is $\bbZ[\frac{1}{n}]$-linear.
\end{example}
\begin{example}[Quadrics]
Let $G=SO_n$. In this case, $\widetilde{G}=\mathrm{Spin}_n$, $\widetilde{Z}\simeq\mu_2$, and $\mathsf{Ch}\simeq\{\pm\}$. Consider the action of $G$ on $\bbP^{n-1}$ by projective linear transformations and let $P \subset G$ be the stabilizer of the isotropic point $[{\bf 1}: 0: \cdots :0]$. The projective homogeneous variety $\cF:=G/P$ is the quadric hypersurface $Q$ in $\bbP^{n-1}$ given by
$
x_1y_1 + \cdots + x_{[n/2]}y_{[n/2]} + z^2 =0$
when $n$ is odd and by
$x_1y_1 + \cdots + x_{[n/2]}y_{[n/2]} =0$
when $n$ is even. Given a $1$-cocycle $\gamma: \mathfrak{g} \to SO_n(k_{\mathrm{sep}})$ and an additive invariant $E:\dgcat \to \mathsf{D}$, we conclude from Panin \cite[\S10.3]{Panin} and \eqref{eq:iso-main-2} that 
\begin{eqnarray}
E(k)^{\oplus (n-2)} \oplus E(C_0(q)) \stackrel{\sim}{\too} E({}_\gamma Q) && n \,\, \mathrm{odd} \label{eq:quadric1} \\
E(k)^{\oplus (n-2)} \oplus E(C^+_0(q)) \oplus E(C^-_0(q)) \stackrel{\sim}{\too} E({}_\gamma Q) && n \,\, \mathrm{even} \label{eq:quadric2} \,,
\end{eqnarray}
where $q$ is the quadric form associated to $\gamma$ and $C_0^+(q), C_0^-(q)$ the simple components of the even Clifford algebra $C_0(q)$. Since the degree of the Clifford algebras is a power of $2$, we conclude from Theorem \ref{thm:coefficients} that the left-hand-side of \eqref{eq:quadric1} (resp. of \eqref{eq:quadric2}) reduces to $\bigoplus_{i=1}^{n-1} E(k)$ (resp. to $\bigoplus_{i=1}^{n} E(k)$) whenever $\mathsf{D}$ is $\bbZ[\frac{1}{2}]$-linear.
\end{example}
\begin{example}[Forms of quadrics]
Let $G=PSO_n$ with $n$ even. Given a $1$-cocycle $\gamma: \textgoth{g} \to PSO_n(k_{\mathrm{sep}})$ and an additive invariant $E:\dgcat \to \mathsf{D}$, we conclude from \cite[\S10.3]{Panin} and \eqref{eq:iso-main-2} that
\begin{equation}\label{eq:form1}
(\bigoplus_{\underset{\mathrm{even}}{i>0}}^{n-3} E(k)) \oplus (\bigoplus_{\underset{\mathrm{odd}}{i>0}}^{n-3} E(A)) \oplus E(C_0^+(A,\sigma)) \oplus E(C_0^-(A,\sigma)) \stackrel{\sim}{\too} E({}_\gamma Q)\,,
\end{equation}
where $(A,\sigma)$ is the algebra with involution associated to $\gamma$. Since $A$ is of degree $n$, we conclude from Theorem~\ref{thm:coefficients} that the left-hand-side of \eqref{eq:form1} reduces to $\bigoplus_{i=1}^n E(k)$ whenever $\mathsf{D}$ is $\bbZ[\frac{1}{2}]$-linear.
\end{example}
\subsection*{Quasi-split case}
As explained in Remark~\ref{rk:quasi-split}, when $G$ is a quasi-split algebraic group the algebras $A_{\chi(i),\gamma}$ are only separable. In this generality, we also have the following complete computation (which generalizes Theorem~\ref{thm:computation}).
\begin{theorem}\label{thm:computation-1}
Under the assumptions of Theorem~\ref{thm:main}(i) (with $G$ quasi-split), we have the following isomorphisms:
$$
\begin{array}{ll}
\bigoplus_i C(k) \otimes  HH_0(A_{\chi(i),\gamma}) \simeq C({}_\gamma \cF) & \\
\bigoplus _i HH(k) \otimes HH_0(A_{\chi(i),\gamma}) \simeq HH({}_\gamma \cF) & \bigoplus _i HC(k) \otimes HH_0(A_{\chi(i),\gamma}) \simeq HC({}_\gamma \cF) \\
\bigoplus _i HP(k) \otimes HH_0(A_{\chi(i),\gamma}) \simeq HP({}_\gamma \cF) & \bigoplus _i HN(k) \otimes HH_0(A_{\chi(i),\gamma}) \simeq HN({}_\gamma \cF)
\end{array}
$$
where $HH_0(A_{\chi(i),\gamma})$ is the $k$-vector space $A_{\chi(i),\gamma}/[A_{\chi(i),\gamma},A_{\chi(i),\gamma}]$.
\end{theorem}
Intuitively speaking, Theorem~\ref{thm:computation-1} shows us that in this generality cyclic homology (and all its variants) not only measures the index of the Weyl groups but also the ``noncommutativity'' of the Tits' algebras.
\subsection*{Toric varieties}
By applying Theorem~\ref{thm:main}(ii) to the above examples (i)-(viii) of additive invariants we obtain several (concrete) direct summands. In the case of cyclic homology (and all its variants) we have the following computation.
\begin{theorem}\label{thm:computation-direct}
Under the assumptions of Theorem~\ref{thm:main}(ii), we have:
$$
\begin{array}{lcl}
C(X) \,\text{d.s.}\, C(k) \otimes HH_0(A) && \\
HH(X)  \,\text{d.s.}\, HH(k) \otimes HH_0(A) && HC(X)  \,\text{d.s.}\, HC(k) \otimes HH_0(A) \\
HP(X)  \,\text{d.s.}\, HP(k)\otimes HH_0(A) && HN(X)  \,\text{d.s.}\, HN(k) \otimes HH_0(A) \,,
\end{array}
$$
where $HH_0(A)=\bigoplus_\rho (A_\rho/[A_\rho, A_\rho])$ and $\text{d.s.}$ stands for ``direct summand''. 
\end{theorem}
\begin{remark}
When $S$ is a {\em quasi-split} reductive group (\eg\ an algebraic torus), the algebras $A_\rho$ are commutative; see \cite[Remark~4.3]{MP}.
\end{remark}
\begin{example}[Toric models]\label{example:toric}
Let $T$ be an algebraic torus and $X$ a smooth projective toric $T$-model. As proved by Merkurjev-Panin in \cite[Prop.~5.6]{MP}, one can construct out of this data an algebraic torus $S$, an $S$-torsor $\pi:U \to X$, and an $S$-equivariant open embedding of $U$ into an affine space on which $S$ acts linearly. As a consequence, Theorem~\ref{thm:main}(ii) and the above Theorem~\ref{thm:computation-direct} hold in these cases. In what concerns toric varieties, please consult Remark~\ref{rk:toric}.
\end{example}
To the best of the author's knowledge, all the computations obtained in \S\ref{sec:results}\text{-}\ref{sec:applications} are new in the literature.
\medbreak\noindent\textbf{Notations:}
We will reserve the letter $k$ for the base field, the letters $X,Y,Z$ for smooth projective $k$-schemes, and the letters $A,B,C$ for separable $k$-algebras. Given a small category $\cC$, we will write $\mathrm{Iso}\, \cC$ for the set of isomorphism classes of objects. Finally, (unless stated differently) all tensor products will be taken over $k$. 
\medbreak\noindent\textbf{Acknowledgments:} The author is very grateful to Asher Auel for discussions about twisted projective homogeneous varieties, to Christian Haesemeyer for discussions about toric varieties, to Bjorn Poonen for discussions about the Brauer group, and to Alexander Merkurjev for pointing out his beautiful work with Panin \cite{MP}. The author would like also to thank Mikhail Kapranov and Yale's Department of Mathematics for their hospitality.

\section{Differential graded categories}\label{sec:dg}
Let $\cC(k)$ be the category of complexes of $k$-vector spaces. A {\em differential graded (=dg) category $\cA$} is a category enriched over $\cC(k)$. A {\em dg functor} $F:\cA\to \cB$ is  a functor enriched over $\cC(k)$; consult Keller's ICM survey \cite{ICM-Keller} for details. In what follows, we will write $\dgcat$ for the category of (small) dg categories and dg functors.

Let $\cA$ be a dg category. The category $\dgHo(\cA)$ has the same objects as $\cA$ and $\dgHo(\cA)(x,y):=H^0\cA(x,y)$. The {\em opposite} dg category $\cA^\op$ has the same objects as $\cA$ and $\cA^\op(x,y):=\cA(y,x)$. A {\em right $\cA$-module} is a dg functor $\cA^\op \to \cC_\dg(k)$ with values in the dg category $\cC_\dg(k)$ of complexes of $k$-vector spaces. Let us write $\cC(\cA)$ for the category of right $\cA$-modules. As explained in \cite[\S3.1]{ICM-Keller}, the dg structure of $\cC_\dg(k)$ makes $\cC(\cA)$ into a dg category $\cC_\dg(\cA)$. Recall from \cite[\S3.2]{ICM-Keller} that the {\em derived category $\cD(\cA)$ of $\cA$} is the localization of $\cC(\cA)$ with respect to quasi-isomorphisms. Its subcategory of compact objects will be denoted by $\cD_c(\cA)$.

A dg functor $F:\cA\to \cB$ is called a {\em derived Morita equivalence} if the restriction of scalars functor $\cD(\cB) \stackrel{\sim}{\to} \cD(\cA)$ is an equivalence. As proved in \cite[Thm.~5.3]{IMRN}, $\dgcat$ admits a Quillen model structure whose weak equivalences are the derived Morita equivalences. Let $\Hmo$ be the associated homotopy category.

The {\em tensor product $\cA\otimes\cB$} of two dg categories $\cA$ and $\cB$ is defined as follows: the set of objects is the cartesian product of the sets of objects of $\cA$ and $\cB$ and $(\cA\otimes\cB)((x,w),(y,z)):= \cA(x,y) \otimes \cB(w,z)$. As explained in \cite[\S2.3]{ICM-Keller}, this construction gives rise to a symmetric monoidal category $(\dgcat, -\otimes-, k)$.

Given dg categories $\cA$ and $\cB$, an {\em $\cA\text{-}\cB$-bimodule $\mathsf{B}$} is a dg functor $\mathsf{B}:\cA \otimes \cB^\op\to \cC_\dg(k)$, \ie a right $(\cA^\op \otimes \cB)$-module. Standard examples are the $\cA\text{-}\cA$-bimodule
\begin{eqnarray}\label{eq:bimodule-Id}
\cA \otimes \cA^\op \too \cC_\dg(k) && (x,y) \mapsto \cA(y,x)
\end{eqnarray}
and more generally the $\cA\text{-}\cB$-bimodule
\begin{eqnarray}\label{eq:bimodules111}
{}_F\mathsf{B}:\cA\otimes \cB^\op \too \cC_\dg(k) && (x,w) \mapsto \cB(w,F(x))
\end{eqnarray}
associated to a dg functor $F:\cA \to \cB$.
\begin{notation}
Given dg categories $\cA$ and $\cB$, let $\rep(\cA,\cB)$ be the full triangulated subcategory of $\cD(\cA^\op \otimes \cB)$ consisting of those $\cA\text{-}\cB$-bimodules $\mathsf{B}$ such that $\mathsf{B}(x,-) \in \cD_c(\cB)$ for every object $x \in \cA$. In the same vein, let $\rep_\dg(\cA,\cB)$ be the full dg subcategory of $\cC_\dg(\cA^\op \otimes \cB)$ consisting of those $\cA\text{-}\cB$-bimodules which belong to $\rep(\cA,\cB)$. By construction, we have $\dgHo(\rep_\dg(\cA,\cB))\simeq \rep(\cA,\cB)$.
\end{notation}
\subsection*{Kontsevich's smooth and proper dg categories}
Recall from Kontsevich \cite{IAS,Miami,finMot} that a dg category $\cA$ is called {\em smooth} if the above $\cA\text{-}\cA$-bimodule \eqref{eq:bimodule-Id} belongs to $\cD_c(\cA^\op \otimes \cA)$ and {\em proper} if for each ordered pair of objects $(x,y)$ we have $\sum_i \mathrm{dim}\, H^i \cA(x,y) < \infty$. The standard examples are the finite dimensional $k$-algebras of finite global dimension (when $k$ is perfect) and the dg categories $\perf_\dg(X)$ associated to smooth projective $k$-schemes $X$. As proved in \cite[Thm.~5.8]{CT1}, the smooth and proper dg categories can be characterized as being precisely the rigid (or dualizable) objects of the symmetric monoidal category $(\Hmo, -\otimes - , k)$. Moreover, the dual of $\cA$ is its opposite dg category $\cA^\op$. This gives rise (\cite[Lemma~5.9]{CT1}) to the following equivalence and derived Morita equivalence
\begin{eqnarray}\label{eq:reps}
\rep(\cA,\cB) \simeq \cD_c(\cA^\op \otimes \cB) & & \rep_\dg(\cA,\cB) \simeq \cA^\op \otimes \cB\,.
\end{eqnarray}
\section{Noncommutative (Chow) motives}\label{sec:NCmotives}
Let $\cA$ and $\cB$ be two dg categories. As proved in
\cite[Cor.~5.10]{IMRN}, we have a natural bijection $\Hom_{\Hmo}(\cA,\cB)\simeq \mathrm{Iso}\,\rep(\cA,\cB)$ under which the composition law of $\Hmo$ corresponds to the derived tensor product of bimodules 
\begin{eqnarray}\label{eq:bimodules11}
\rep(\cA,\cB) \times \rep(\cB,\cC) \too\rep(\cA,\cC) && (\mathsf{B},\mathsf{B}')\mapsto \mathsf{B} \otimes^\bfL_\cB \mathsf{B}'\,.
\end{eqnarray}
Moreover, the identity of an object $\cA$ corresponds to the isomorphism class of the $\cA\text{-}\cA$-bimodule \eqref{eq:bimodule-Id}. Since the above $\cA\text{-}\cB$-bimodules \eqref{eq:bimodules111} clearly belong to
$\rep(\cA,\cB)$, we have a well-defined $\otimes$-functor
\begin{eqnarray}\label{eq:functor1}
\dgcat \too \Hmo && F \mapsto {}_F\mathsf{B}\,.
\end{eqnarray}
The {\em additivization} of $\Hmo$ is the additive symmetric monoidal category $\Hmo_0$ with the same objects as $\Hmo$ and with abelian groups of morphisms given by $\Hom_{\Hmo_0}(\cA,\cB):=K_0\rep(\cA,\cB)$. The composition law is induced from the above bi-triangulated functor \eqref{eq:bimodules11} and the symmetric monoidal structure by bilinearity from $\Hmo$. We have also a well-defined $\otimes$-functor
\begin{eqnarray}\label{eq:nat2}
\Hmo \too \Hmo_0 && \mathsf{B} \mapsto [\mathsf{B}]\,.
\end{eqnarray}
As proved in \cite[Thms.~5.3 and 6.3]{IMRN}, the composition
$$ U:\dgcat \stackrel{\eqref{eq:functor1}}{\too} \Hmo \stackrel{\eqref{eq:nat2}}{\too} \Hmo_0$$
is the {\em universal additive invariant}, i.e. given any additive category $\mathsf{D}$ there is an induced equivalence of categories
\begin{equation}\label{eq:categories}
U^\ast: \Fun_{\add}(\Hmo_0,\mathsf{D}) \stackrel{\sim}{\too} \Fun_{\mathsf{A}}(\dgcat,\mathsf{D})\,,
\end{equation}
where the left-hand-side denotes the category of additive functors and the right-hand-side the category of additive invariants in the sense of Definition \ref{def:additive}. Because of this universal property, which is reminiscent from motives, $\Hmo_0$ is called the (additive\footnote{A triangulated version also exists in the literature; see \cite{Duke}.}) category of {\em noncommutative motives}; consult the survey article \cite{survey}. Since the functors \eqref{eq:functor1}\text{-}\eqref{eq:nat2} are the identity on objects, we will often make no notational distinction between a dg category and its image in $\Hmo_0$.
\subsection*{Kontsevich's noncommutative Chow motives}
Kontsevich introduced in \cite{IAS,Miami,finMot} the category $\NChow$ of {\em noncommutative Chow motives}. This category identifies with smallest full additive subcategory of $\NChow$ generated by the smooth and proper dg categories. Note that $\NChow$ is a rigid symmetric monoidal category. Moreover, thanks to the left-hand-side of \eqref{eq:reps}, we have the isomorphisms
\begin{equation}\label{eq:Hom-NChow}
 \Hom_{\NChow}(\cA,\cB) \simeq K_0\cD_c(\cA^\op \otimes \cB)\,.
\end{equation} 
\section{From Merkurjev-Panin to Kontsevich}\label{sec:MP-K}
In this section we construct a fully-faithful $\otimes$-functor $\Theta$ from Merkurjev-Panin's motivic category $\underline{\cC}$ to Kontsevich's category of noncommutative Chow motives $\NChow$; see Theorem~\ref{thm:bridge}. This functor will play a key role in the sequel.
\begin{notation}
In what follows we will write $\SmProj$ for the category of smooth projective $k$-schemes and $\Sep$ for the category of separable $k$-algebras. Given $X,Y,Z \in \SmProj$, the projection map from $X \times Y \times Z$ to $X \times Y, X\times Z, Y \times Z$ will be denoted by $p^{XYZ}_{XY}, p^{XYZ}_{XZ}, p^{XYZ}_{YZ}$, respectively.
\end{notation}
\subsection*{Merkurjev-Panin's motivic category $\underline{\cC}$}
Recall from \cite[\S1]{MP} the construction of the category $\underline{\cC}$. The objects are the pairs $(X,A)$ with $X \in \SmProj$ and $A \in \Sep$. The morphisms are given by the Grothendieck groups
$$\Hom_{\underline{\cC}}((X,A),(Y,B)) := K_0 \mathrm{vect}(X\times Y,A^\op \otimes B) \,,$$
where $\mathrm{vect}(X\times Y,A^\op \otimes B)$ is the exact category of those right $(\cO_{X \times Y} \otimes (A^\op \otimes B))$-modules which are locally free and of finite rank as $\cO_{X\times Y}$-modules. Given $[\cF] \in K_0 \mathrm{vect}(X\times Y,A^\op \otimes B)$ and $[\cF'] \in K_0 \mathrm{vect}(Y\times Z,B^\op \otimes C)$, their composition $[\cF']\circ [\cF]$ is defined as 
$$ (p^{XYZ}_{XZ})_\ast \left((p^{XYZ}_{XY})_\ast(\cF) \otimes_B (p^{XYZ}_{YZ})_\ast(\cF')\right) \in K_0 \mathrm{vect}(X\times Z,A^\op \otimes C)\,,$$
where the direct image $(p^{XYZ}_{XZ})_\ast$ is defined only at the $K_0$-theoretical level; consult \cite[\S1.3]{MP} for details. The identity of an object $(X,A)$ is the class $[\cO_\Delta \otimes A] \in K_0 \mathrm{vect}(X\times X,A^\op \otimes A)$, where $\Delta$ is the diagonal of $X \times X$. The category $\underline{\cC}$ comes equipped with a symmetric monoidal structure
\begin{eqnarray*}
\underline{\cC} \times \underline{\cC} \too \underline{\cC} & ((X,A), (Y,B)) \mapsto (X \times Y, A \otimes B)
\end{eqnarray*}
and with two $\otimes$-functors
\begin{eqnarray*}
\Phi:\SmProj^\op \too \underline{\cC} && \Psi: \Sep \too \underline{\cC}\,.
\end{eqnarray*}
The (contravariant) functor $\Phi$ sends $X$ to the pair $(X,k)$ and a map $f:X \to Y$ to $[\cO_{\Gamma^t_f}] \in K_0\mathrm{vect}(Y \times X)$, where $\Gamma^t_f$ stands for the transpose of the graph $\Gamma_f:=\{(x,f(x))\,|\,x \in X\} \subset X \times Y$ of $f$. On the other hand, the functor $\Psi$ sends $A$ to $(\mathrm{Spec}(k),A)$ and a $k$-algebra homomorphism $h:A \to B$ to $[{}_h \mathsf{B}] \in K_0(A^\op \otimes B)$.
\subsection*{Perfect complexes}
Given $X \in \SmProj$ and $A \in \Sep$, let $\Mod(\cO_X\otimes A)$ be the Grothendieck category of right $(\cO_X \otimes A)$-modules, $\cD(\cO\otimes A):= \cD(\Mod(\cO_X \otimes A))$ its derived category, and $\perf(X,A)$ the full subcategory of $\cD(\cO_X\otimes A)$ consisting of those complexes of right $(\cO_X \otimes A)$-modules which are perfect as complexes of $\cO_X$-modules. Note that $\perf(X,A) = \rep(A^\op, \perf_\dg(X))$.
\begin{lemma}\label{lem:K0s}
The canonical inclusion of categories $\mathrm{vect}(X,A) \subset \perf(X,A)$ gives rise to an isomorphism $K_0 \mathrm{vect}(X,A) \simeq K_0\perf(X,A)$ of abelian groups.
\end{lemma}
\begin{proof}
Let $\mathrm{Coh}(X,A)$ be the abelian category of those right $(\cO_X \otimes A)$-modules which are coherent as $\cO_X$-modules. As explained in \cite[\S1.1]{MP}, Quillen's resolution theorem implies that the canonical inclusion $\mathrm{vect}(X,A) \subset \mathrm{Coh}(X,A)$ give rise to an isomorphism $K_0\mathrm{vect}(X,A) \simeq K_0\mathrm{Coh}(X,A)$. Making use of it one defines
\begin{eqnarray}\label{eq:Euler}
K_0\perf(X,A) \too K_0\mathrm{vect}(X,A) && \cG \mapsto \sum_i (-1)^i H^i(\cG)\,.
\end{eqnarray}
A simple verification shows that \eqref{eq:Euler} is the inverse of the induced homomorphism $K_0\mathrm{vect}(X,A) \to K_0\perf(X,A)$. This achieves the proof.
\end{proof}
Let $\cE$ be an abelian category. As explained in \cite[\S4.4]{ICM-Keller}, the derived category $\cD_\dg(\cE)$ of $\cE$ is the dg quotient $\cC_\dg(\cE)/\cA c_\dg(\cE)$ of the dg category of complexes over $\cE$ by its full dg subcategory of acyclic complexes. In what follows, we will write $\cD_\dg(\cO_X \otimes A)$ for the dg category $\cD_\dg(\cE)$, with $\cE:= \Mod(\cO_X \otimes A)$, and $\perf_\dg(X,A)$ for the full dg subcategory of those complexes of right $(\cO_X \otimes A)$-modules which belong to $\perf(X,A)$. By construction, we have $\dgHo(\cD_\dg(\cO_X \otimes A)) \simeq \cD(\cO_X \otimes A)$ and $\dgHo(\perf_\dg(X,A)) \simeq \perf(X,A)$.
\begin{lemma}\label{lem:smooth}
The dg category $\perf_\dg(X,A)$ is smooth and proper.
\end{lemma}
\begin{proof}
Let us start by showing that $A$, considered as a dg category $A$, is smooth and proper. Since $k$ is a field, $A$ is a finite dimensional $k$-algebra; see Knus-Ojanguren \cite[\S III Prop.~3.2]{Knus}. This implies properness. As explained in \cite[\S III Thm.~1.4]{Knus}, $A$ is separable if and only if $A$ is projective as a $A\text{-}A$-bimodule. This implies smoothness. Making use of the right-hand-side of \eqref{eq:reps}, one then obtains the following derived Morita equivalence
\begin{equation}\label{eq:Morita-eq}
\perf_\dg(X,A) = \rep_\dg(A^\op, \perf_\dg(X)) \simeq A \otimes \perf_\dg(X)\,.
\end{equation}
The proof follows now from the fact that $\perf_\dg(X)$ is smooth and proper.
\end{proof}
\begin{proposition}[Projection formula]\label{prop:projective}
Let $f:X \to Y$ be a flat proper morphism in $\SmProj$ and $A,B,C \in \Sep$. Under these assumptions, we have canonical quasi-isomorphisms of dg functors
\begin{eqnarray*}
(\cG,\cG') & \mapsto & (\bfR f)_\ast(\cG) \otimes^\bfL_B \cG' \too (\bfR f)_\ast (\cG \otimes^\bfL_B f^\ast(\cG'))\\
(\cH,\cH') & \mapsto &\cH' \otimes^\bfL_B (\bfR f)_\ast(\cH) \too (\bfR f)_\ast (f^\ast(\cH) \otimes^\bfL_B \cH)\,,
\end{eqnarray*}
where $\cG \in \perf_\dg(X,A^\op \otimes B), \cG' \in \perf_\dg(Y,B^\op \otimes C), \cH \in \perf_\dg(X,B^\op \otimes C)$ and $\cH' \in \perf_\dg(Y,A^\op \otimes B)$.
\end{proposition}
\begin{proof}
The proof is similar to the one of Thomason-Trobaugh~\cite[Prop.~3.17]{TT}.
\end{proof}
\begin{proposition}[Base-change formula]\label{prop:base-change}
Let $A$ be a separable $k$-algebra and 
$$ 
\xymatrix{
X' \ar[d]_-{f'} \ar[r]^-{g'} \ar@{}[dr]|{\ulcorner}& X \ar[d]^-f \\
Y' \ar[r]_-g & Y
}
$$
a cartesian square in $\SmProj$ with $f$ flat proper and $g$ flat. Under these assumptions, we have a canonical quasi-isomorphism of dg functors 
\begin{eqnarray*}
\perf_\dg(X,A) \ni\cG & \mapsto & g^\ast (\bfR f)_\ast(\cG) \to (\bfR f')_\ast (g')^\ast(\cG)\,.
\end{eqnarray*}
\end{proposition}
\begin{proof}
The proof is similar to the one of Thomason-Trobaugh~\cite[Prop.~3.18]{TT}.
\end{proof}
Given $X, Y \in \SmProj$ and $A,B \in \Sep$, every $\cG \in \perf(X\times Y, A^\op \otimes B)$ gives rise to the following Fourier-Mukai dg functor
\begin{eqnarray*}
\Phi_\cG: \perf_\dg(X,A) \too \perf_\dg(Y,B) && \cE \mapsto (\bfR p^{XY}_Y)_\ast\left((p^{XY}_X)^\ast(\cE) \otimes_A^\bfL \cG\right)\,.
\end{eqnarray*}
\begin{lemma}\label{lem:key1}
We have a well-defined triangulated functor
\begin{eqnarray*}
\perf(X\times Y, A^\op \otimes B) \too \rep(\perf_\dg(X,A),\perf_\dg(Y,B)) && \cG \mapsto {}_{\Phi_\cG}\mathsf{B}\,.
\end{eqnarray*}
\end{lemma}
\begin{proof}
Every morphism $\eta: \cG \to \cG'$ of complexes of right $(\cO_{X \times Y} \otimes (A^\op \otimes B))$-modules gives rise to a morphism of dg functors $\Phi_\eta: \Phi_\cG \Rightarrow \Phi_{\cG'}$ and consequently to a morphism of bimodules ${}_{\Phi_\eta}\mathsf{B}: {}_{\Phi_\cG} \mathsf{B} \Rightarrow {}_{\Phi_{\cG'}}\mathsf{B}$. Whenever $\alpha$ is a quasi-isomorphism, $\dgHo(\Phi_\eta)$ is a natural isomorphism between triangulated functors. Using \cite[Lemma~9.8]{MT}, one then concludes that ${}_{\Phi_\eta}\mathsf{B}$ is a quasi-isomorphism of bimodules. This implies that the functor is well-defined. The fact that it is triangulated is clear by now.
\end{proof}
\subsection*{The functor $\Theta$ from $\underline{\cC}$ to $\NChow$.}
Let  $X,Y \in \SmProj$ and $A,B \in \Sep$. By combining Lemmas \ref{lem:K0s} and \ref{lem:key1}, one obtains the following homomorphism
\begin{equation}\label{eq:theta-maps}
K_0 \mathrm{vect}(X\times Y, A^\op \otimes B) \stackrel{[\cF] \mapsto [{}_{\Phi_\cF}\mathsf{B}]}{\too} K_0 \rep(\perf_\dg(X,A), \perf_\dg(Y,B)) \,.
\end{equation}
\begin{theorem}\label{thm:bridge}
The assignments $(X,A) \mapsto \perf_\dg(X,A)$ and $[\cF] \mapsto [{}_{\Phi_\cF}\mathsf{B}]$ give rise to a fully-faithful $\otimes$-functor $\Theta: \underline{\cC} \to \NChow$ making the diagrams commute
$$
\xymatrix{
\SmProj^\op \ar[d]_-\Phi \ar[rr]^-{X \mapsto \perf_\dg(X)} && \dgcat \ar[d]^-U & \Sep \ar[d]_-\Psi \ar[rr]^-{A \mapsto A} && \dgcat \ar[d]^-U \\
\underline{\cC} \ar[rr]_-\Theta && \NChow \subset \Hmo_0  & \underline{\cC} \ar[rr]_-\Theta && \NChow \subset \Hmo_0\,.
}
$$
\end{theorem}
\subsection*{Proof of Theorem~\ref{thm:bridge}}
We start by showing that $\Theta$ preserves the identities.
\begin{lemma}\label{lem:units}
For every $X \in \SmProj$ and $A \in \Sep$, the class $[{}_{\Phi_{(\cO_\Delta \otimes A)}}\mathsf{B}]$ agrees with the identity of $\perf_\dg(X,A)$ in $\NChow$.
\end{lemma}
\begin{proof}
Let $\mathrm{Id}:\perf_\dg(X,A) \to \perf_\dg(X,A)$ be the identity dg functor. Since the identity of $\perf_\dg(X,A)$ in $\NChow$ is the class $[{}_{\mathrm{Id}}\mathsf{B}]$, one needs to show that
\begin{eqnarray}\label{eq:equality-classes}
[{}_{\Phi_{(\cO_\Delta \otimes A)}}\mathsf{B}]=[{}_{\mathrm{Id}} \mathsf{B}] & \text{in} & K_0\rep(\perf_\dg(X,A),\perf_\dg(X,A))\,.
\end{eqnarray}
Let $\iota$ be the composition $X \stackrel{\sim}{\to} \Delta \subset X \times X$ and $p$ (resp. $q$) the projection map from $X \times X$ to the first (resp. second) component.
Under these notations, we have the following natural quasi-isomorphisms
\begin{eqnarray}
\Phi_{(\cO_\Delta \otimes A)}(\cE) & := & (\bfR q)_\ast \left(p^\ast(\cE) \otimes^\bfL_A (\cO_\Delta \otimes A) \right) \nonumber \\
& \simeq & (\bfR q)_\ast \left(p^\ast(\cE) \otimes^\bfL_A (\bfR \iota)_\ast(\cO_X \otimes A) \right) \nonumber \\
& \simeq & (\bfR q)_\ast \left((\bfR \iota)_\ast(\iota^\ast(p^\ast(\cE))\otimes^\bfL_A(\cO_X \otimes A)) \right)\label{eq:star-1} \\
& \simeq & (\bfR q)_\ast \left((\bfR \iota)_\ast(\iota^\ast(p^\ast(\cE))) \right)\\ \label{eq:star-2}
& \simeq & (\bfR q \bfR \iota)_\ast ((p \iota)^\ast(\cE)) \simeq \cE\,. \label{eq:star-3}
\end{eqnarray}
Some explanations are in order: \eqref{eq:star-1} follows from Proposition~\ref{prop:projective} (with $f=\iota$ and $B,C$ equal to $A$), \eqref{eq:star-2} follows from the fact that $\cO_X \otimes A$ is the $\otimes$-unit of the symmetric monoidal dg category $\perf_\dg(X,A^\op \otimes A)$; and finally \eqref{eq:star-3} follows from the equalities $q\iota\simeq \id$ and $p \iota \simeq \id$. We obtain in this way a quasi-isomorphism $\Phi_{{\cO_\Delta \otimes A}} \Rightarrow \mathrm{Id}$ of dg functors. As in the proof of Lemma~\ref{lem:key1}, we conclude that the bimodules ${}_{\Phi_{(\cO_\Delta \otimes A)}}\mathsf{B}$ and ${}_{\mathrm{Id}}\mathsf{B}$ are quasi-isomorphic. This implies the above equality \eqref{eq:equality-classes} and so the proof is finished.
\end{proof}
Given $\cG \in \perf(X\times Y,A^\op \otimes B)$ and $\cG' \in \perf(Y \times Z, B^\op \otimes C)$,
consider the following perfect complex 
$$ \cG \star \cG' :=(\bfR p_{XYZ}^{XZ})_\ast\left((p_{XYZ}^{XY})^\ast(\cG)\otimes^\bfL_B(p_{XYZ}^{YZ})^\ast(\cG')\right) \in \perf(X \times Z, A^\op \otimes C)\,.$$
\begin{lemma}\label{lem:auxiliar} 
Under the above notations, we have the following equality
\begin{eqnarray*}
[{}_{\Phi_{\cG \star \cG'}}\mathsf{B}] =[{}_{\Phi_\cG}\mathsf{B} \otimes^\bfL_{\perf_\dg(Y,B)} {}_{\Phi_{\cG'}}\mathsf{B}] & \text{in} & K_0\rep(\perf_\dg(X,A),\perf_\dg(Z,C))\,.
\end{eqnarray*}
\end{lemma}
\begin{proof}
Recall from \S\ref{sec:dg}\text{-}\ref{sec:NCmotives} that $[{}_{\Phi_\cG}\mathsf{B} \otimes^\bfL_{\perf_\dg(Y,B)} {}_{\Phi_{\cG'}}\mathsf{B}]= [{}_{\Phi_{\cG'} \circ \Phi_\cG}\mathsf{B}]$. The proof will consist on showing that ${}_{\Phi_{\cG \star \cG'}}\mathsf{B}$ and ${}_{\Phi_{\cG'}\circ \Phi_{\cG}}\mathsf{B}$ are quasi-isomorphic, which implies automatically the above equality. We have the following natural quasi-isomorphisms
\begin{eqnarray}
\Phi_{\cG \star \cG'}(\cE) &:= & (\bfR p_Z^{XZ})_\ast \left((p_X^{XZ})^\ast(\cE) \otimes^\bfL_A (\bfR p^{XYZ}_{XZ})_\ast((p^{XYZ}_{XY})^\ast (\cG) \otimes^\bfL_B (p^{XYZ}_{YZ})^\ast(\cG')) \right) \nonumber \\
&& (\bfR p^{YZ}_Z)_\ast (\bfR p^{XYZ}_{YZ})_\ast \left((p^{XYZ}_{XY})^\ast ((p^{XYZ}_{XY})^\ast(\cE) \otimes^\bfL_B \cG) \otimes^\bfL_B (p^{XYZ}_{YZ})^\ast (\cG') \right) \label{eq:star1-1-1} \\
&& (\bfR p^{YZ}_Z)\left((\bfR p^{XYZ}_{YZ})_\ast (p^{XYZ}_{XY})^\ast ((p^{XYZ}_{XY})^\ast(\cE) \otimes^\bfL_A \cG) \otimes^\bfL_B \cG'\right) \label{eq:star2-2-2} \\
& \simeq & (\bfR p^{YZ}_Z)_\ast \left((p_Y^{YZ})^\ast (\bfR p^{XY}_Y)_\ast ((p^{XYZ}_{XY})^\ast(\cE) \otimes^\bfL_A \cG) \otimes^\bfL_B \cG' \right) \label{eq:star3-3-3} \\
&=& (\bfR p^{YZ}_Z)_\ast \left((p^{YZ}_Y)^\ast \Phi_\cG(\cE) \otimes^\bfL_B \cG' \right)= \Phi_{\cG'} (\Phi_\cG(\cE)) = \Phi_{\cG' \circ \cG}(\cE)\,. \nonumber
\end{eqnarray}
Some explanations are in order: \eqref{eq:star1-1-1} follows from Proposition~\ref{prop:projective} (with $f=p^{XYZ}_{XZ}$) and from the equalities $p^{XZ}_Z p^{XYZ}_{XZ} = p^{XYZ}_Z, p^{XZ}_X p^{XYZ}_{XZ} = p^{XY}_X p^{XYZ}_{XY}$ and $p^{XYZ}_Z = p^{YZ}_Z p^{XYZ}_{YZ}$; \eqref{eq:star2-2-2} follows from Proposition~\ref{prop:projective} (with $f=p^{XYZ}_{XY}$); and finally \eqref{eq:star3-3-3} follows from Proposition~\ref{prop:base-change} applied to the cartesian square
$$
\xymatrix{
X \times Y \times Z \ar[r]^-{p^{XYZ}_{YZ}} \ar[d]_-{p^{XYZ}_{XY}} \ar@{}[dr]|{\ulcorner} & Y \times Z \ar[d]^-{p^{YZ}_Y} \\
X \times Y \ar[r]_-{p^{XY}_Y} & Y\,.
}
$$
We obtain in this way a zig-zag of quasi-isomorphisms of dg functors between $\Phi_{\cG \star \cG'}$ and $ \Phi_{\cG' \circ \cG}$. Similarly to the proof of Lemma~\ref{lem:units}, we conclude that the bimodules ${}_{\Phi_{\cG \star \cG'}}\mathsf{B}$ and ${}_{\Phi_{\cG'}\circ \Phi_{\cG}}\mathsf{B}$ are quasi-isomorphic. This achieves the proof.
\end{proof}
Let us now show that $\Theta$ preserves the composition law.
\begin{lemma}\label{lem:composition}
Given $\cF \in \mathrm{vect}(X \times Y, A^\op \otimes B)$ and $\cF' \in \mathrm{vect}(Y \times Z, B^\op \otimes C)$, the image of $[\cF']\circ [\cF] \in K_0 \mathrm{vect}(X \times Z,A^\op\otimes C)$ under the above homomorphism \eqref{eq:theta-maps} agrees with the class $[{}_{\Phi_\cG}\mathsf{B} \otimes^\bfL_{\perf_\dg(Y,B)} {}_{\Phi_{\cG'}}\mathsf{B}]$.
\end{lemma}
\begin{proof}
Note that under the isomorphism of Lemma~\ref{lem:K0s}
$$ K_0 \mathrm{vect}(X\times Z, A^\op \otimes C) \stackrel{\sim}{\too} K_0 \perf(X \times Z, A^\op \otimes C)\,,$$
the composition $[\cF']\circ [\cF]$ agrees with $[\cF\star \cF']$. This implies that the image of $[\cF'] \circ [\cF]$ under the homomorphism \eqref{eq:theta-maps} agrees with $[{}_{\Phi_{\cF \star \cF'}}\mathsf{B}]$. Thanks to Lemma~\ref{lem:auxiliar}, this latter class is equal to $[{}_{\Phi_\cG}\mathsf{B} \otimes^\bfL_{\perf_\dg(Y,B)} {}_{\Phi_{\cG'}}\mathsf{B}]$ and so the proof is finished. 
\end{proof}
The above Lemmas~\ref{lem:units} and \ref{lem:composition} imply that the functor $\Theta$ is well-defined. Let us now show that it is fully-faithful.
\begin{lemma}
The above homomorphisms \eqref{eq:theta-maps} are isomorphisms.
\end{lemma}
\begin{proof}
Recall first from Lemma~\ref{lem:K0s} that the inclusion $\mathrm{vect}(X \times Y, A^\op \otimes B) \subset \perf(X \times Y, A^\op \otimes B)$ gives rise to an isomorphism between $K_0\mathrm{vect}(X \times Y, A^\op \otimes B)$ and $K_0 \perf(X \times Y, A^\op \otimes B)$. The proof follows from the sequence of isomorphisms
\begin{eqnarray}
K_0\perf(X \times Y, A^\op \otimes B) & = & K_0 \rep(A \otimes B^\op, \perf_\dg(X \times Y)) \nonumber \\
& \simeq & K_0 \cD_c((A \otimes \perf_\dg(X))^\op\otimes (B \otimes \perf_\dg(Y))) \label{eq:star1111} \\
& \stackrel{\eqref{eq:reps}}{\simeq} & K_0 \rep(A \otimes \perf_\dg(X), B \otimes \perf_\dg(Y))  \nonumber \\
& \stackrel{\eqref{eq:Morita-eq}}{\simeq} & K_0 \rep(\perf_\dg(X,A), \perf_\dg(Y,B))\,,\nonumber
\end{eqnarray}
where \eqref{eq:star1111} is a consequence of the derived Morita equivalence
\begin{equation}\label{eq:boxtimes}
\boxtimes : \perf_\dg(X) \otimes \perf_\dg(Y) \stackrel{\sim}{\too} \perf_\dg(X \times Y)
\end{equation}
(established in \cite[Prop.~6.2]{Regularity}) and from the fact that $\perf_\dg(X)^\op \simeq \perf_\dg(X)$. 
\end{proof}
By combining the derived Morita equivalences \eqref{eq:Morita-eq} and \eqref{eq:boxtimes}, we conclude that 
$$
\perf_\dg(X,A) \otimes \perf_\dg(Y,B) \simeq \perf_\dg(X \times Y, A \otimes B)
$$
for every $X,Y \in \SmProj$ and $A,B \in \Sep$, \ie that $\Theta$ is symmetric monoidal. It remains then only to show that the diagrams of Theorem~\ref{thm:bridge} are commutative.
\begin{lemma}
For every morphism $f: X \to Y$ in $\SmProj$, we have the equality
\begin{eqnarray}\label{eq:equality-diag-1}
[{}_{\Phi_{\cO_{\Gamma_f^t}}}\mathsf{B}] = [{}_{f^\ast}\mathsf{B}] & \mathrm{in} & K_0\rep(\perf_\dg(Y), \perf_\dg(X))\,.
\end{eqnarray}
\end{lemma}
\begin{proof}
Let $\iota$ be the composition $X \stackrel{x \mapsto (f(x),x)}{\too} \Gamma_f^t \subset Y \times X$. Under this notation, we have the following natural quasi-isomorphisms
\begin{eqnarray}
\Phi_{\cO_{\Gamma_f^t}}(\cE) & := & (\bfR p^{XY}_X)_\ast ((p^{XY}_Y)^\ast(\cE) \otimes \cO_{\Gamma_f^t}) \nonumber \\
& \simeq & (\bfR p^{XY}_X)_\ast \left((p^{XY}_Y)^\ast(\cE) \otimes (\bfR \iota)_\ast(\cO_X) \right) \nonumber \\
& \simeq & (\bfR p^{XY}_X)_\ast \left((\bfR \iota)_\ast (\iota^\ast((p^{XY}_Y)^\ast(\cE)) \otimes \cO_X )\right) \label{eq:star111} \\
& \simeq & (\bfR p^{XY}_X)_\ast \left((\bfR \iota)_\ast (\iota^\ast (p^{XY}_Y)^\ast(\cE))\right)  \label{eq:star222} \\
& \simeq & (\bfR p^{XY}_X \iota)\ast \left((p^{XY}_Y \iota)^\ast (\cE) \right) \simeq f^\ast(\cE) \label{eq:star333}\,.
\end{eqnarray}
Some explanations are in order: \eqref{eq:star111} follows from Proposition~\ref{prop:projective} (with $f=\iota$ and $A,B,C$ equal to $k$), \eqref{eq:star222} follows from the fact that $\cO_X$ is the $\otimes$-unit of the symmetric monoidal dg category $\perf_\dg(X)$; and finally \eqref{eq:star333} follows from the equalities $p^{XY}_X \iota = \id$ and $p^{XY}_Y\iota =f$. We obtain in this way a quasi-isomorphism $\Phi_{\cO_{\Gamma_f^t}} \Rightarrow f^\ast$ of dg functors. As in the proof of Lemma~\ref{lem:units}, we conclude that the bimodules ${}_{\cO_{\Gamma^t_f}}\mathsf{B}$ and ${}_{f^\ast}\mathsf{B}$ are quasi-isomorphisms. This implies the above equality \eqref{eq:equality-diag-1} and so the proof is finished.
\end{proof}
\begin{lemma}\label{lem:diagram-2}
The right-hand-side diagram of Theorem~\ref{thm:bridge} is commutative.
\end{lemma}
\begin{proof}
Given a morphism $h:A \to B$ in $\Sep$, consider the Fourier-Mukai dg~functor
\begin{equation}\label{eq:bimodule-final}
\Phi_{{}_h\mathsf{B}}: \perf_\dg(\mathrm{Spec}(k),A) \to \perf_\dg(\mathrm{Spec}(k),B)\,.
\end{equation}
Making use of the derived Morita equivalence \eqref{eq:Morita-eq} (with $X=\mathrm{Spec}(k)$), one observes that the bimodule associated to \eqref{eq:bimodule-final} is isomorphic in the homotopy category $\Hmo$ to the $A\text{-}B$-bimodule ${}_h \mathsf{B} \in \rep(A,B)$. This achieves the proof.
\end{proof}
\section{Proof of Theorem~\ref{thm:main}}
\subsection*{Item (i)}
Since by hypothesis $E:\dgcat \to \mathsf{D}$ is an additive invariant, the equivalence of categories \eqref{eq:categories} furnish us a (unique) additive functor $\overline{E}$ making the following diagram commute
\begin{equation}\label{eq:com-diag}
\xymatrix{
\dgcat \ar[d]_-U \ar[r]^-E & \mathsf{D} \\
\Hmo_0 \ar[ur]_-{\overline{E}} & .
}
\end{equation}
Now, recall from \cite[\S6]{Panin} the construction of the category $\bbA^G$ and of the functors
\begin{eqnarray*}
\Phi: \SmProj^\op \too \bbA^G && \Psi: \Sep \too \bbA^G\,.
\end{eqnarray*}
In the particular case where $G$ is the trivial group $\{1\}$, $\bbA^{\{1\}}$ identifies with $\underline{\cC}$ and $\Phi,\Psi$ with the functors described above.  As proved by Panin in \cite[page~557]{Panin} (after applying the functor $F_\gamma: \bbA^G \to \bbA^{\{1\}}$), every $\mathsf{Ch}$-homogeneous basis $\rho_1, \ldots, \rho_n$ of $R(\widetilde{P})$ over $R(\widetilde{G})$ gives rise to an isomorphism in $\bbA^{\{1\}}$
\begin{equation}\label{eq:iso-Panin-aux}
\bigoplus^n_{i=1} \Psi(A_{\chi(i),\gamma})\simeq \Psi(\prod_{i=1}^n A_{\chi(i),\gamma}) \stackrel{\sim}{\too} \Phi({}_\gamma \cF)\,.
\end{equation}
Note that by construction of the category of noncommutative motives we have $\bigoplus_{i=1}^n U(A_{\chi(i),\gamma})= U(\prod^n_{i=1} A_{\chi(i),\gamma})$. Hence, by combining Theorem \ref{thm:bridge} with the above commutative diagram \eqref{eq:com-diag}, we conclude that the image of \eqref{eq:iso-Panin-aux} under
\begin{equation}\label{eq:composed}
\bbA^{\{1\}} = \underline{\cC} \stackrel{\Theta}{\too} \NChow \subset \Hmo_0 \stackrel{\overline{E}}{\too} \mathsf{D}
\end{equation}
agrees with the desired isomorphism \eqref{eq:iso-main-2}. This achieves the proof.
\subsection*{Item (ii)}
As proved in Merkurjev-Panin in \cite[Thm.~4.2]{MP}, $\Phi(X)$ is a direct summand of $\Psi(A)$ in the motivic category $\underline{\cC}$. Using Theorem~\ref{thm:bridge}, the above commutative diagram \eqref{eq:com-diag}, and the above composed functor \eqref{eq:composed}, we then conclude that $E(X)$ is a direct summand of $E(A)$ in the additive category $\mathsf{D}$.
\begin{remark}[Toric varieties]\label{rk:toric}
Let $T$ be an algebraic torus and $X$ a smooth projective toric $T$-variety. As proved by Merkurjev-Panin in \cite[Thm.~7.6]{MP} (making use of the splitting principle), there exists a separable $k$-algebra $B$ and a retraction of $\Psi(B)$ into $\Phi(X)$ in the motivic category $\underline{\cC}$. The above proof of item (ii) shows that Theorem~\ref{thm:main}(ii) and Theorem~\ref{thm:computation-direct} also hold in these cases.
\end{remark}
\section{Cyclic homology of separable algebras}\label{sec:computations}
\begin{theorem}\label{thm:degree}
Given a separable $k$-algebra $A$, we have:
$$
\begin{array}{lcl}
C(A) \simeq C(k)\otimes HH_0(A) && \\
HH(A) \simeq HH(k)\otimes HH_0(A) && HC(A) \simeq HC(k)\otimes HH_0(A) \\
HP(A) \simeq HP(k)\otimes HH_0(A) && HN(A) \simeq HN(k) \otimes HH_0(A)\,. 
\end{array}
$$
\end{theorem}
As a consequence of Theorem~\ref{thm:degree} one obtains the isomorphisms:
\begin{eqnarray*}
HH_j(A) \simeq \left\{ \begin{array}{ll}
A/[A,A] &   j=0 \\
0 &  \text{otherwise} \\
\end{array} \right. && 
HC_j(A) \simeq \left\{ \begin{array}{ll}
A/[A,A] & j\geq 0 \,\, \text{even}  \\
0 &  \text{otherwise} \\
\end{array} \right.
\end{eqnarray*}
\begin{eqnarray*}
HP_j(A) \simeq \left\{ \begin{array}{ll}
A/[A,A] &  j \,\,\text{even} \\
0 & j \,\,\text{odd} \\
\end{array} \right. && 
HN_j(A) \simeq \left\{ \begin{array}{ll}
A/[A,A] &  j\leq 0 \,\, \text{even}  \\
0 &  \text{otherwise} \\
\end{array} \right.\,.
\end{eqnarray*}
\begin{proof}
Recall from Loday \cite[\S1.1]{Loday} the construction of the Hochschild homology complexes $HH(A)$ and $HH(k)$. In particular, the $k$-vector spaces in degree $j$ are given by $A^{\otimes (j+1)}$ and $k^{\otimes (j+1)}$, respectively. Note that the assignments
\begin{eqnarray*}
A^{\otimes (j+1)} \too k^{\otimes (j+1)} \otimes A/[A,A] && a_0 \otimes \cdots \otimes a_j \mapsto 1 \otimes \cdots \otimes 1 \otimes (a_1 \cdots a_j)
\end{eqnarray*}
gives rise to a well-defined homomorphism of complexes
\begin{equation}\label{eq:homo-key}
HH(A) \too HH(k) \otimes A/[A,A]\,.
\end{equation}
By combining \cite[Thm.~1.2.13]{Loday} with \cite[Cor.~1.2.14]{Loday} (and their proofs), one observes that \eqref{eq:homo-key} is a quasi-isomorphism. Since $A/[A,A] = HH_0(A)$, we conclude then that $HH(A)$ is isomorphic to $HH(k)\otimes A/[A,A]$ in the derived category $\cD(k)$. Now, recall from \cite[\S2.5]{Loday} that $HH(A)$ and $HH(k) \otimes A/[A,A]$ are endowed with the following cyclic action
\begin{eqnarray}
t(a_0 \otimes \cdots \otimes a_j) & := & (-1)^j (a_n \otimes a_0 \otimes \cdots \otimes a_{j-1}) \nonumber\\
t(\lambda_0 \otimes \cdots \otimes \lambda_j \otimes a) & := & (-1)^j (\lambda_j \otimes \lambda_0 \otimes \cdots \otimes \lambda_{j-1} \otimes a)\,,
\label{eq:action2}
\end{eqnarray}
where $t$ stands for the generator of $\bbZ/(j+1)\bbZ$. The above homomorphism \eqref{eq:homo-key} preserves this cyclic action. As a consequence, following \cite[\S2.5]{Loday}, one obtains well-defined quasi-isomorphisms:
\begin{eqnarray}
C(A) & \too& \mathrm{Mix}(HH(k) \otimes HH_0(A))\nonumber \\
HC(A) & \too& \mathrm{Tot}(B(HH(k) \otimes HH_0(A))) \label{eq:star-final-2}\\
HP(A) & \too& \mathrm{Tot}(B^{\mathrm{per}}(HH(k) \otimes HH_0(A))) \label{eq:star-final-3}\\
HN(A) & \too& \mathrm{Tot}(B^-(HH(k) \otimes HH_0(A))) \label{eq:star-final-4}\,.
\end{eqnarray}
Some explanations are in order: $\mathrm{Mix}(-)$ stands for the mixed complex construction, $B(-), B^{\mathrm{per}}(-), B^-(-)$ for the cyclic, periodic and negative bicomplex constructions, and $\mathrm{Tot}(-)$ for the totalization functor. Thanks to the above action \eqref{eq:action2}, one observes that $\mathrm{Mix}(HH(k) \otimes HH_0(A))$ identifies with the mixed complex $HH_0(A) \otimes A/[A,A]$. As a consequence, $C(A)$ is isomorphic to $C(k) \otimes HH_0(A)$ in the derived category of mixed complexes $\cD(\Lambda)$. In what concerns \eqref{eq:star-final-2}, we have a canonical isomorphism between its right-hand-side and the complex $\mathrm{Tot}(B(HH(k)))\otimes A/[A,A]=HC(k) \otimes A/[A,A]$. This implies that $HC(A)$ is isomorphic to $HC(k) \otimes HH_0(A)$ in the derived category $\cD(k)$. Finally, in what concerns \eqref{eq:star-final-3}-\eqref{eq:star-final-4}, we have canonical quasi-isomorphisms
\begin{equation*}
\mathrm{Tot} (B^{\mathrm{per}}(HP(k))) \otimes A/[A,A] \too \mathrm{Tot} (B^{\mathrm{per}}(HP(k) \otimes A/[A,A])) 
\end{equation*}
\begin{equation*}
\mathrm{Tot} (B^-(HN(k))) \otimes A/[A,A] \too \mathrm{Tot} (B^-(HN(k) \otimes A/[A,A])) \,.
\end{equation*}
Making use of them, we conclude that $HP(A)$ (resp. $HN(A)$) is isomorphic to $HP(k) \otimes HH_0(k)$ (resp. to $HN(k) \otimes HH_0(k)$) in the derived category $\cD_{\bbZ/2}(k)$ (resp. $\cD(k)$). This achieves the proof.

\end{proof}
\begin{proposition}\label{prop:zero}
When $A$ is a central simple $k$-algebra, we have $HH_0(A)\simeq k$.
\end{proposition}
\begin{proof}
As explained in \cite[\S III Thm.~5.13]{Knus}, we have a $\otimes$-equivalence of categories 
\begin{eqnarray}\label{eq:eq}
\Mod(k) \stackrel{\sim}{\too} \Mod(A^\op \otimes A) && M \mapsto M \otimes A\,.
\end{eqnarray}
Since $k$ is a field, \cite[Prop.~1.1.13]{Loday} implies the following isomorphisms
\begin{eqnarray}\label{eq:Tor}
HH_\ast(k) \simeq \text{Tor}_n^{\Mod(k)}(k,k) && HH_\ast(A) \simeq \text{Tor}_n^{\Mod(A^\op \otimes A)}(A,A)\,.
\end{eqnarray}
By combining \eqref{eq:eq}\textrm{-}\eqref{eq:Tor}, we then conclude that $HH_0(A) \simeq HH_0(k)\simeq k$.
\end{proof}
\subsection*{Proof of Theorem \ref{thm:computation}}
The proof follows from the combination of isomorphism \eqref{eq:iso-main-2} (with $E=C, HH, HC, HP$ and $HN$) with Theorems \ref{thm:degree} and \ref{prop:zero}.
\subsection*{Proof of Theorem \ref{thm:coefficients}}
Given a central simple $k$-algebra $A$ of degree $r$, \cite[Cor.~3.1]{TV} implies that $E(k) \simeq E(A)$ whenever $\mathsf{D}$ is $\bbZ[\frac{1}{r}]$-linear. The proof follows then from the combination of this fact with isomorphism \eqref{eq:iso-main-2}.
\subsection*{Proof of Theorem \ref{thm:computation-1}}
The proof follows from the combination of isomorphism \eqref{eq:iso-main-2} (with $E=C, HH,HC,HP$ of $HN$) with Theorem~\ref{thm:degree}.
\subsection*{Proof of Theorem \ref{thm:computation-direct}}
The proof follows from the combination of Theorem~\ref{thm:main}(ii) (with $E=C,HH,HC,HP$ and $HN$) with Theorem~\ref{thm:degree}.
\section{Noncommutative motives of central simple algebras}
Let us denote by $\mathrm{Br}(k)$ the Brauer group of $k$.
\begin{theorem}\label{thm:Brauer-aux}
Given central simple $k$-algebras $A$ and $B$, we have:
\begin{eqnarray*}
U(A) \simeq U(B) \Leftrightarrow [A]=[B] \in \mathrm{Br}(k)\,.
\end{eqnarray*} 
\end{theorem}
\begin{proof}
We start with some preliminaries. Given a central simple $k$-algebra $C$, let us denote by $\mathsf{P}(C)$ the exact category of finitely generated projective right $C$-modules. As explained by Gille-Szamuely in \cite[Thm.~2.1.3]{Gille}, the Wedderburn theorem implies that $C \simeq M_{n \times n}(D)$ for some integer $m \geq 1$ and division algebra $D \supset k$. As a consequence, the right $C$-modules $M \in \mathsf{P}(C)$ are completely classified by their length. In particular, every element of the Grothendieck group $K_0\mathsf{P}(C)\simeq \bbZ$ can be written as $\pm[M]$ for a unique (up to isomorphism) right $C$-module $M \in \mathsf{P}(C)$.

Now, let us assume that $[A]=[B] \in \mathrm{Br}(k)$. This equality is equivalent to the existence of bimodules $\mathsf{B} \in \mathsf{P}(A^\op \otimes B)$ and $\mathsf{B}' \in \mathsf{P}(B^\op \otimes A)$ such that $\mathsf{B} \otimes_B \mathsf{B}' \simeq A$ in $\mathsf{P}(A^\op \otimes A)$ and $\mathsf{B}' \otimes_A\mathsf{B} \simeq B$ in $\mathsf{P}(B^\op \otimes B)$. In other words, the central simple $k$-algebras $A$ and $B$ are Morita equivalent. Making use of the canonical inclusions of categories $\mathsf{P}(A^\op \otimes B) \subset \cD_c(A^\op \otimes B)$ and $\mathsf{P}(B^\op \otimes A) \subset \cD_c(B^\op \otimes A)$, 
one then concludes that $A$ and $B$ are isomorphic in the homotopy category $\Hmo$. This implies automatically that $U(A) \simeq U(B)$ and so ($\Leftarrow$) is proved. 

Let us now prove ($\Rightarrow$). Assume that $U(A)\simeq U(B)$. By construction of Kontsevich's category of noncommutative Chow motives $\NChow$, there exist elements $\alpha \in K_0 \cD_c(A^\op \otimes B)$ and $\beta \in K_0 \cD_c(B^\op \otimes A)$ such that $\beta \circ \alpha =[A]$ in $K_0 \cD_c(A^\op \otimes A)$ and $\alpha \circ \beta = [B]$ in $K_0\cD_c(B^\op \otimes B)$. Making use of the isomorphisms
\begin{eqnarray*}
K_0 \mathsf{P}(A^\op \otimes B) \simeq K_0 \cD_c(A^\op \otimes B) && K_0 \mathsf{P}(B^\op \otimes A) \simeq K_0 \cD_c(B^\op \otimes A) \\
K_0 \mathsf{P}(A^\op \otimes A) \simeq K_0 \cD_c(A^\op \otimes A)  && K_0 \mathsf{P}(B^\op \otimes B) \simeq K_0 \cD_c(B^\op \otimes B) \,,
\end{eqnarray*}
one then concludes from above preliminaries that $\alpha = \pm [M]$ for a unique $A\text{-}B$-bimodule $M$ in $\mathsf{P}(A^\op \otimes B)$, that $\beta= \pm [N]$ for a unique $B\text{-}A$-bimodule $N$ in $\mathsf{P}(B^\op \otimes B)$, that $M \otimes_B N \simeq A$ in $\mathsf{P}(A^\op \otimes A)$, and that $N\otimes_A M \simeq B$ in $\mathrm{P}(B^\op \otimes B)$. This implies the equality $[A]=[B] \in \mathrm{Br}(k)$ and so the proof is finished.
\end{proof}
\begin{proposition}\label{prop:Brauer1}
Given central simple $k$-algebras $A_1, \ldots, A_n$, we have 
\begin{equation}\label{eq:sum}
\bigoplus_{i=1}^n U(A_i) \simeq \bigoplus_{i=1}^n U(k)
\end{equation} 
if and only if $[A_1] = \cdots = [A_n]=[k] \in \mathrm{Br}(k)$.
\end{proposition}
\begin{proof}
Assume first that $[A_1]=\cdots = [A_n]=[k]$. Making use of Theorem~\ref{thm:Brauer-aux}, we conclude that $U(A_i)\simeq U(k)$ for every $i$ and consequently that \eqref{eq:sum} holds. This proves ($\Leftarrow$). Let us now prove the converse implication. Assume \eqref{eq:sum}. By combining isomorphism \eqref{eq:Hom-NChow} (with $\cA=\cB=k$) with the linearity of $\NChow$ and with the fact that $k$ is a field, we observe that 
\begin{equation}\label{eq:ring}
\mathrm{End}_{\NChow}(\oplus_{i=1}^n U(k)) \simeq M_{n \times n}(\bbZ)\,.
\end{equation}
Let $\{e_i\}_{i=1}^n$ (resp. $\{\overline{e_i}\}_{i=1}^n$) be the idempotent elements of $\mathrm{End}_{\NChow}(\bigoplus_{i=1}^n U(A_i))$ (resp. of $\mathrm{End}_{\NChow}(\bigoplus_{i=1}^n U(k))$) associated to the $n$ factors. Via \eqref{eq:sum}, the idempotents $\{e_i\}_{i=1}^n$ correspond to idempotent elements $\{\widetilde{e_i}\}_{i=1}^n$ of $\mathrm{End}_{\NChow}(\bigoplus_{i=1}^n U(k))$. Using \eqref{eq:ring} and the fact that $\sum_i \widetilde{e_i} =\id$, we conclude that $\widetilde{e_i}=\overline{e_j}$ for some $j$. This implies that $U(A_i)=U(k)$ for every $i$. Consequently, using Theorem~\ref{thm:Brauer-aux}, we obtain the equalities $[A_1]= \cdots =[A_n] =[k]$. This concludes the proof.
\end{proof}
\subsection*{Proof of Theorem \ref{thm:Brauer}}
The proof follows from the combination of isomorphism \eqref{eq:iso-main-2} (with $E=U$) with Proposition~\ref{prop:Brauer1}.

\end{document}